\documentclass[12pt]{l4dc2023} 

\usepackage{times}
\usepackage{amsmath, amssymb,amsfonts,dsfont,mathtools, nicematrix}
\usepackage{multirow}
\usepackage{graphicx,adjustbox,color}
\usepackage{booktabs}
\usepackage[font=small]{caption}
\usepackage{enumitem}
\usepackage{hyperref}
\usepackage{colortbl}

\definecolor{red}{rgb}{1,0,0}

\usepackage{stackengine}

\title[Distributionally Robust Lyapunov Function Search Under Model Uncertainty]{Distributionally Robust Lyapunov Function Search Under Uncertainty}

\author{\Name{Kehan Long} \Email{k3long@ucsd.edu}\\
 \Name{Yinzhuang Yi} \Email{yiyi@ucsd.edu}\\
 \Name{Jorge Cort{\'e}s} \Email{cortes@ucsd.edu}\\
 \Name{Nikolay Atanasov} \Email{natanasov@ucsd.edu}\\
 \addr Contextual Robotics Institute, University of California San Diego, La Jolla, CA 92093, USA}

\newcommand{\calD}{{\cal D}}

\newcommand{\calF}{{\cal F}}

\newcommand{\calL}{{\cal L}}
\newcommand{\calM}{{\cal M}}
\newcommand{\calN}{{\cal N}}

\newcommand{\calP}{{\cal P}}

\newcommand{\calU}{{\cal U}}

\newcommand{\calX}{{\cal X}}

\newcommand{\bfx}{\mathbf{x}}

\newcommand{\bfz}{\mathbf{z}}

\newcommand{\bftheta}{\boldsymbol{\theta}}

\newcommand{\bfxi}{\boldsymbol{\xi}}

\newcommand{\bbE}{\mathbb{E}}

\newcommand{\bbN}{\mathbb{N}}

\newcommand{\bbP}{\mathbb{P}}
\newcommand{\bbQ}{\mathbb{Q}}
\newcommand{\bbR}{\mathbb{R}}

\begin{document}

\maketitle

\begin{abstract}%
This paper develops methods for proving Lyapunov stability of dynamical systems subject to disturbances with an unknown distribution. We assume only a finite set of disturbance samples is available and that the true online disturbance realization may be drawn from a different distribution than the given samples. We formulate an optimization problem to search for a sum-of-squares (SOS) Lyapunov function and introduce a distributionally robust version of the Lyapunov function derivative constraint. We show that this constraint may be reformulated as several SOS constraints, ensuring that the search for a Lyapunov function remains in the class of SOS polynomial optimization problems. For general systems, we provide a distributionally robust chance-constrained formulation for neural network Lyapunov function search. Simulations demonstrate the validity and efficiency of either formulation on non-linear uncertain dynamical systems. 



\end{abstract}


\subsection*{Supplementary Material}
Open-source implementation available at \href{https://github.com/KehanLong/DR-Lyapunov-Function}{https://github.com/KehanLong/DR-Lyapunov-Function}.

\vspace*{-1ex}
\section{Introduction}\label{sec: intro}

A Lyapunov function (LF) is one of the main tools for analyzing the stability of nonlinear dynamical systems \citep{khalil1996nonlinear}.
Similarly, control synthesis for open-loop control-affine systems is often done using a control Lyapunov function (CLF) \citep{Artstein1983StabilizationWR} since a stabilizing controller can be obtained from a CLF using a universal formula \citep{SONTAG1989117} or quadratic programming \citep{Galloway2015TorqueSI}.
Various techniques exist for obtaining LF or CLF candidates but the majority assume that the system model is known. In this paper, we study the problem of synthesizing a Lyapunov function when the system model is uncertain.





Synthesizing a valid LF for a linear system can be formulated as a semi-definite program (SDP) \citep{Boyd_LMI_control}. \cite{parrilo2000structured, Papachristo_2002_sos_lf} generalized the formulation for non-linear polynomial systems by using sum-of-squares (SOS) polynomials to represent an LF candidate. For polynomial systems with uncertainty, \cite{Ahmadi_2016robust_sos, Lasserre2015TractableAO} extended SOS techniques to find robust LFs based on known error bounds, see \cite{Laurent2009_sos} for 
a general exposition.
The lack of a valid SOS LF does not imply that the system instability since there exist positive-definite functions that are not representable as SOS \citep{HilbertUeberDD}. 



Using a neural network as a more general LF representation than an SOS polynomial has been gaining increasing popularity. \cite{Richards2018TheLN} proposed a neural network approach for discrete-time non-linear systems that learns the region of attraction of a given controller. \cite{Boffi2020LearningSC, gaby2021lyapunov_net} improved the efficiency of learning LFs by incorporating positive-definiteness and equilibrium conditions directly into the network architecture. \cite{Chang2019NeuralLC, dai_2021_lyapunov} considered learning neural network CLFs and controllers by minimizing violations of the conditions for a valid LF. \cite{dawson_2022_robust_CLBF} extended the idea to learn safety certificates as control Lyapunov-barrier functions and also considered control-affine systems with convex-hull uncertainty. The survey by \cite{Dawson_2022_survey} provides a recent account of this line of research.  Despite their expressiveness, neural network methods do not offer theoretical guarantees for the validity of the learned Lyapunov function over the entire state space.
%
%

%

The stability guarantees provided by either SOS or neural network LFs are sensitive to uncertainty or disturbances in the system model. 
A related body of work \citep{Choi2020ReinforcementLF, Taylor_2019iros,Castaneda_GPCLF_ACC21, dhiman2020control, Long2022RAL} assumes that an LF certificate is given for a nominal system and develops approaches to adapt it by taking the model uncertainty into account during deployment. 
In input-to-state stability (ISS) \citep{iss_sontag}, one deals directly with an uncertain dynamical system to provide robustness guarantees on graceful degradation of stability as a function of the disturbance input magnitude. This can be ensured via an ISS-Lyapunov function \citep{sontag1995characterizations}, e.g., constructed  using SOS techniques \citep{ISS_LF_joao,ISS_LF_SOS_rick}. Without assuming any known distribution or error bounds on the model uncertainty, in this paper, we utilize distributionally robust constraints \citep{AS-DD-AR:14, Esfahani2018DatadrivenDR} to enforce LF conditions for an uncertain system model with only finitely many uncertainty samples obtained offline.



Distributionally robust chance-constrained programming (DRCCP) deals with uncertain variables in the constraints using finitely many available samples. The main idea is to construct an ambiguity ball centered at the empirical distribution of the observed samples using a distribution distance function, such as Wasserstein distance \citep{Esfahani2018DatadrivenDR,Hota2019DataDrivenCC}. Then, the constraints are required to be satisfied with high probability for all distributions in the ambiguity ball. Given its powerful guarantee to handle uncertainty with an unknown or shifting distribution, DRCCP has been applied in several areas in systems and control \citep{Coulson_19_cdc, Boskos_21_TAC, long2022_clf_cbf_drccp, coppens_2020_l4dc}. However, its application to LF search here is novel.

\textbf{Contributions}: 1) We formulate a distributionally robust version of the Lyapunov function derivative constraint for uncertain dynamical systems using finitely many offline samples. 2) For polynomial systems, we show that the distributionally robust constraint can be reformulated as multiple SOS constraints, ensuring that LF synthesis with uncertainty remains an SOS polynomial optimization. 3) For general nonlinear systems, we propose a distributionally robust neural network approach for learning Lyapunov functions.





\vspace*{-1ex}
\section{Background}\label{sec: prelim}
Here\footnote{The sets of non-negative real and natural numbers are denoted $\bbR_{\geq 0}$ and $\bbN$. For $N \in \bbN$, $[N] := \{1,2, \dots N\}$. We denote the distribution and expectation of a random variable $Y$ by $\mathbb{P}$ and $\bbE_{\bbP}(Y)$, resp. 
For a scalar $x$, $(x)_+ := \max(x,0)$.  We use $\boldsymbol{0}$ to denote the $n$-dimensional vector with all entries equal to $0$. The gradient of a differentiable function $V$ is denoted by $\nabla V$, and its Lie derivative along a vector field $f$ by $\calL_f V  = \nabla V \cdot f$. 
We denote a uniform distribution on $[a,b]$ as $\calU(a,b)$ and a Gaussian distribution with mean $\mu$ and variance $\sigma^2$ as~$\calN(\mu,\sigma^2)$.}, we give an overview of sum-of-squares (SOS) techniques for Lyapunov function synthesis and distributionally robust chance constraints.

\vspace*{-1ex}
\subsection{Lyapunov Theory and Sum-of-Squares Optimization}
\label{sec: sos_lyapunov_original}
Consider a dynamical system, $\dot{\bfx} = f(\bfx)$, with state $\bfx \in \calX \subseteq \bbR^n$. Assume $f : \mathbb{R}^{n} \mapsto \mathbb{R}^{n}$ is locally Lipschitz and the origin $\bfx = \mathbf{0}$ is the desired equilibrium , i.e., $f(\mathbf{0}) = \mathbf{0}$.
%
A valid Lyapunov function, ensuring the stability of the origin, satisfies:
\begin{equation}
\label{eq: original_lf_condition}
    V(\boldsymbol{0}) = 0,  \; V(\bfx) > 0 \; \text{and} \; \dot{V}(\bfx) < 0, \; \forall \bfx \neq \boldsymbol{0},
\end{equation}
%
where $\dot{V}(\bfx) = \calL_f V(\bfx)$.
If the LF is also radially unbounded ($V(\bfx) \to \infty$ as $\|\bfx\| \to \infty$), then its existence implies global asymptotic stability. The second and third conditions in~\eqref{eq: original_lf_condition}
are implied by
\begin{equation}
\label{eq: lf_to_sos}
    \quad V(\bfx) - \epsilon\|\bfx\|_2^2 \geq 0 \; \text{and} \; -\dot{V}(\bfx) - \epsilon\|\bfx\|_2^2 \geq 0, \; \forall \bfx \neq \boldsymbol{0},
\end{equation}
for some $\epsilon \in \bbR_{>0}$.
A natural way of imposing non-negativity is by using SOS polynomials. A polynomial $\eta(\bfx)$ of degree $2d$ is called an SOS polynomial if and only if there exist polynomials $s_1(\bfx), \dots, s_p(\bfx)$ of degree at most $d$ such that $\eta(\bfx) = \sum_{i=0}^{p}s_i(\bfx)^2$. Based on the positive-definiteness property of SOS polynomials, \citet{parrilo2000structured, Papachristo_2002_sos_lf} proposed the following SOS conditions, which are sufficient to imply \eqref{eq: original_lf_condition},
\begin{equation}
\label{eq: original_lf_sos}
     V(\bfx) = \sum_{k=0}^{2d} c_k \bfx^k, \; c_0 = 0; \quad V(\bfx) - \epsilon\|\bfx\|_2^2 \in \text{SOS}(\bfx); \; \;
     -\dot{V}(\bfx) - \epsilon\|\bfx\|_2^2 \in \text{SOS}(\bfx),
\end{equation}
where $\text{SOS}(\bfx)$ denotes the set of SOS polynomials in variable $\bfx$. By fixing a polynomial degree $d$, one can search for an SOS LF using a semidefinite program \citep{Laurent2009_sos} enforcing~\eqref{eq: original_lf_sos}.

\vspace*{-1ex}
\subsection{Conditional Value-at-Risk and Distributionally Robust Chance Constraint }\label{sec: prelim_cvar_drccp}


We review chance-constraint formulations that will be useful to handle model uncertainty. Consider a complete separable metric space $\Xi$ with metric $d$, and associate to it a Borel $\sigma$-algebra $\calF$ and the set $\calP(\Xi)$ of Borel probability measures on $\Xi$.
A chance constraint can be written as,  
\begin{equation}\label{eq: ccp}
    \bbP^*(G(\bfz, \bfxi) \leq 0) \geq 1 - \beta, 
\end{equation}
where the constraint function $G(\bfz, \bfxi) \in \mathbb{R}^n \times \Xi \mapsto \mathbb{R}$ depends both on a decision vector $\bfz$ and a random variable $\bfxi$ with distribution $\bbP^* \in \calP(\Xi)$, and $\beta \in (0,1)$ is a user-specified risk tolerance. The feasible set for $\bfz$ defined by \eqref{eq: ccp} is not convex. \eqref{eq: ccp}, \cite{Nemirovski2006ConvexAO} proposed a Conditional Value-at-Risk (CVaR) approximation of the chance constraint, which results in a convex feasible set and is sufficient for \eqref{eq: ccp} to hold:
\begin{equation}\label{eq: cvar_ccp}
    \textrm{CVaR}_{1-\beta}^{\bbP^*}(G(\bfz,\bfxi)) \leq 0.
\end{equation}
%
%
For a random variable $\xi \in \bbR$ with distribution $\hat{\bbP}$, the Value-at-risk (VaR) at confidence level $1 - \beta$ is  $\textrm{VaR}_{1-\beta}^{\hat{\bbP}}(\xi) := \inf_{t \in \bbR}\{t \; | \; \hat{\bbP}(\xi \leq t) \geq 1 - \beta\}$. The CVaR of $\xi$ is  $\textrm{CVaR}_{1-\beta}^{\hat{\bbP}}(\xi) := \bbE_{\hat{\bbP}} [ \xi \; | \; \xi \geq \textrm{VaR}_{1-\beta}^{\hat{\bbP}}(\xi)]$ and can be formulated as a convex program \citep{Rockafellar00optimizationof}: 
\begin{equation}
\label{eq: cvar_opti_def}   
    \textrm{CVaR}_{1-\beta}^{\hat{\bbP}}(\xi) = \inf_{t \in \mathbb{R}}[\beta^{-1}\mathbb{E}_{\hat{\bbP}}[(\xi+t)_+]-t].
\end{equation}


The chance constraint in \eqref{eq: ccp} or \eqref{eq: cvar_ccp} cannot be specified if the distribution $\bbP^*$ of $\bfxi$ is unknown. In robotics and control applications, it is common that only finitely many samples $\{\bfxi_i\}_{i=1}^N$ from $\bbP^*$ are available. This motivates a distributionally robust formulation of the chance constraint \citep{Esfahani2018DatadrivenDR, Xie2021OnDR}. Let $\calP_p(\Xi) \subseteq \calP(\Xi)$ be the set of Borel probability measures with finite $p$-th moment for $p \geq 1$. The $p$-Wasserstein distance between two probability measures $\mu$, $\nu$ in $\calP_p(\Xi)$ is defined as \citep[see, for example][]{Chen2018DataDrivenCC, Xie2021OnDR}:
%
    $W_{p}(\mu,\nu) := \big(\inf_{\gamma \in \bbQ(\mu,\nu)} \big[ \int_{\Xi \times \Xi} d(x,y)^p \text{d}\gamma(x,y) 
    \big] \big)^{\frac{1}{p}}, $
%
where $\bbQ(\mu,\nu)$ denotes the collection of all measures on $\Xi \times \Xi$ with marginals $\mu$  and $\nu$ on the first and second factors, and $d$ denotes the metric in $\Xi$.

We denote by $\hat{\mathbb{P}}_N :=  \frac{1}{N}\sum_{i=1}^N \delta_{\bfxi_i}$ the discrete empirical distribution of samples $\{\bfxi_i\}_{i=1}^N$. Using the Wasserstein distance, we define an ambiguity set $\calM_{N}^{r} := \{\mu \in \calP_p(\Xi) \; | \; \allowbreak W_p(\mu,\hat{\mathbb{P}}_{N} ) \leq r\}$ as a ball of distributions with radius $r$ centered at $\hat{\mathbb{P}}_N$. We write a distributionally robust version of the chance constraint in \eqref{eq: ccp} as $\inf_{\bbP \in \calM_{N}^{r}}\mathbb{P}(G(\bfz, \bfxi) \leq 0) \geq 1 - \beta$ or equivalently $\sup_{\bbP \in \calM_{N}^{r}}\mathbb{P}(G(\bfz, \bfxi) \geq 0) \leq \beta$. Thus, similar to the CVaR approximation in \eqref{eq: cvar_ccp}-\eqref{eq: cvar_opti_def}, one considers the sufficient constraint
$
    \sup_{\bbP \in \calM_{N}^{r}}\inf_{t \in \mathbb{R}}[\beta^{-1} \mathbb{E}_{\bbP}[(G(\bfz,\bfxi)+t)_{+}] -t] \leq 0
$, which is convex in $\bfz$.

%
%
%
%

%

\vspace*{-1ex}
\section{Problem Formulation}\label{sec: problem_formulation}

We aim to analyze Lyapunov stability for a dynamical system subject to model uncertainty: 
\begin{equation}\label{eq: uncertain_closed_loop_system}
\dot{\bfx} =  f(\bfx) + \sum_{i = 1}^m d_i(\bfx)\xi_i = f(\bfx) + d(\bfx)\bfxi, 
\end{equation}
where $d_i: \bbR^n \mapsto \bbR^n$ is locally Lipschitz. We assume that $d(\bfx) = [d_1(\bfx), \dots, d_m(\bfx)] \in \bbR^{n \times m}$ is known or estimated from state-control trajectories \citep{harrison_meta_learning, thai_2022_learn_disturbance}. We do not assume any known error bounds or distribution for the parameter $\bfxi \in \Xi \subseteq \mathbb{R}^m$. Instead, we consider a finite data set of samples $\{\bfxi_i\}_{i=1}^N$ that may be used for LF synthesis. 
The uncertainty model in \eqref{eq: uncertain_closed_loop_system} captures the commonly considered additive disturbance, which in our formulation corresponds to $m=n$ and $d(\mathbf{x}) = \boldsymbol{I}_n$. The matrix $d(\bfx)$ allows specifying particular system modes affected by the disturbance $\bfxi$ depending on the state $\bfx$.


\begin{problem}[\textbf{Lyapunov Function Search For Uncertain Systems}]\label{prob: LF_Construction_CCP}
Given a finite set of uncertainty samples $\{\bfxi_i\}_{i=1}^N$ from the uncertain system in \eqref{eq: uncertain_closed_loop_system}, obtain a Lyapunov function $V: \bbR^n \mapsto \bbR$ that can be used to verify the stability of the origin while taking the uncertainty into account.
\end{problem}

\vspace*{-1ex}
\section{Lyapunov Function Search For Systems with Model Uncertainty}
\label{sec: approach_LF}


We present an SOS approach (Sec.~\ref{sec: drccp_sos_lf}) and a neural network approach (Sec.~\ref{sec: nn_lf_search}) to address Problem~\ref{prob: LF_Construction_CCP}. Our methodology is based on finding a function $V: \mathbb{R}^n \mapsto \mathbb{R}$ that satisfies the Lyapunov conditions in \eqref{eq: original_lf_condition}. The uncertainty in the dynamical system \eqref{eq: uncertain_closed_loop_system} appears in the term $\dot{V}(\bfx)$, which presents a challenge for ensuring that the condition $\dot{V}(\bfx) < 0$, $\forall \bfx \neq  \boldsymbol{0}$ is satisfied. 

\subsection{Sum-of-Squares Approach For Lyapunov Function Search}\label{sec: drccp_sos_lf}
%
%



We first introduce our SOS approach for LF synthesis under model uncertainty. The Lyapunov conditions in \eqref{eq: lf_to_sos}, taking the uncertainty in \eqref{eq: uncertain_closed_loop_system} into account, become:
\begin{equation}
\label{eq: probabilistic_lf_condition}
    V(\mathbf{0}) = 0;  \; \forall \bfx \neq \boldsymbol{0}, \; V(\bfx) - \epsilon\|\bfx\|_2^2 \geq 0 \; \text{and} \; \bbP^*(-\dot{V}(\bfx,\bfxi) - \epsilon\|\bfx\|_2^2 \geq 0  ) \geq 1-\beta,
\end{equation}
where $\bbP^*$ denotes the true distribution of $\bfxi$. 

To simplify the presentation, let 
%
    $
    G(\bfx,\bfxi) = \dot{V}(\bfx,\bfxi ) + \epsilon \|\bfx\|_2^2 = \nabla V(\bfx)^\top(f(\bfx)+d(\bfx)\bfxi) + \epsilon \|\bfx\|_2^2$, 
%
so that the chance-constraint in \eqref{eq: probabilistic_lf_condition} becomes $\bbP^* (-G(\bfx,\bfxi) \geq 0) \geq 1 - \beta $, $\forall \bfx \neq \boldsymbol{0}$. Based on the discussion in Sec.~\ref{sec: prelim_cvar_drccp}, the CVaR approximation provides a sufficient condition for enforcing the chance constraint: $\inf_{t \in \mathbb{R}}\left[\beta^{-1}\mathbb{E}_{\bbP^*}[(G(\bfx,\bfxi)+t)_{+}] - t \right] \leq 0 $, for all $\bfx \neq \boldsymbol{0}$. If the true distribution $\bbP^*$ were known, this formulation could be used to deal with the uncertainty. However, we are only provided with samples $\{\bfxi_i\}_{i=1}^N$ from $\bbP^*$. We thus rewrite the condition by multiplying by $\beta$ on both sides and using the empirical expectation to approximate the true expectation, 
\begin{equation}
\label{eq: cvar_explicit_form}
    \inf_{t \in \mathbb{R}}\left[\frac{1}{N} \sum_{i=1}^{N} (G(\bfx,\bfxi_i) + t)_+  - t \beta \right] \leq 0 , \quad \forall \bfx \neq \boldsymbol{0} .
\end{equation}
Due to the infimum term in the constraint, one cannot directly write \eqref{eq: cvar_explicit_form} as an SOS condition, as in~\eqref{eq: original_lf_sos}. The following result provides an alternative SOS condition that ensures~\eqref{eq: cvar_explicit_form} holds. 

\begin{proposition}[CC-SOS Condition]
\label{proposition: chance_constraint_relax_lf}
Assume $\beta \leq \frac{1}{N}$, the constraint in \eqref{eq: cvar_explicit_form} is equivalent to:
\begin{equation}\label{eq: proposition_1}
     \max_{i} \beta (\dot{V}(\bfx,\bfxi_i) + \epsilon \|\bfx\|_2^2) \leq 0, \quad \forall \bfx \neq \boldsymbol{0}, 
\end{equation}
Furthermore, if $f$ and $d_i$ are polynomials, the following $N$ SOS conditions are sufficient for \eqref{eq: proposition_1},
\begin{equation} \label{eq: proposition_1_sos}
     -\dot{V}(\bfx,\bfxi_i) - \epsilon\|\bfx\|_2^2 \in \operatorname{SOS}(\bfx), \quad \forall i = 1,2 \dots, N .
\end{equation}
\end{proposition}

%
%

\begin{proof}
Denote by $t^{*}$ the value when the infimum is attained in \eqref{eq: cvar_explicit_form}. Without loss of generality, we assume that for a given $\bfx$, $G(\bfx,\bfxi_i) \geq G(\bfx,\bfxi_j)$, for all $1 \leq i < j \leq N$. Observe that for each $\bfx \neq \bf0$, the function $\frac{1}{N} \sum_{i=1}^{N} (G(\bfx,\bfxi_i) + t)_+  - t \beta$ is piecewise-linear in $t$ with $N+1$ intervals and $N$ breakpoints, given by $\{ -G(\bfx,\bfxi_i)\}_{i=1}^N$ and the slope for the $i$-th interval is $\frac{i-1}{N} - \beta$. Thus, the optimal solution is $t^{*} = -G(\bfx,\bfxi_k)$, where $k$ satisfies $\frac{k-1}{N} - \beta < 0 \;\; \text{and} \;\;  \frac{k}{N} - \beta \geq 0$. The constraint in \eqref{eq: cvar_explicit_form} can be rewritten as $\frac{1}{N} \sum_{i=1}^{k} (G(\bfx,\bfxi_i) - G(\bfx,\bfxi_k)) + \beta G(\bfx,\bfxi_k) \leq 0,\quad \forall \bfx \neq \boldsymbol{0}$. Since $\beta \leq \frac{1}{N}$, only the first interval has negative slope and this constraint can be written as \eqref{eq: proposition_1}. Inspired by the SOS formulation in \eqref{eq: original_lf_sos}, \eqref{eq: proposition_1} is implied by the $N$ SOS constraints in \eqref{eq: proposition_1_sos}.
\end{proof}


%
Using Proposition~\ref{proposition: chance_constraint_relax_lf}, we propose a chance-constrained (CC)-SOS formulation to search for a valid Lyapunov function for the uncertain system in \eqref{eq: uncertain_closed_loop_system}:
%
%
\begin{equation}
\label{eq: ccp_sos_final}
     V(\bfx) = \sum_{k=0}^{2d} c_k \bfx^k, \; c_0 = 0; \;  V(\bfx) - \epsilon\|\bfx\|_2^2 \in \text{SOS}(\bfx); \;
     -\dot{V}(\bfx,\bfxi_i) - \epsilon\|\bfx\|_2^2 \in \text{SOS}(\bfx),
\end{equation}
for all $i \in [N]$. Note that by using CVaR approximations in~\eqref{eq: cvar_explicit_form} and assuming $\beta \leq \frac{1}{N}$, the CC-SOS formulation becomes equivalent to the formulation that is robust against the provided samples $\{\bfxi_i\}_{i=1}^N$, as shown in \eqref{eq: proposition_1_sos}. This CC-SOS formulation overcomes the lack of knowledge of the true uncertainty distribution $\bbP^*$ by using the available samples $\bfxi_i$ to conservatively approximate the probabilistic constraint in \eqref{eq: probabilistic_lf_condition} with $N$ SOS conditions. Nonetheless, the test-time validity of a Lyapunov function satisfying \eqref{eq: ccp_sos_final} is not guaranteed because the CC-SOS condition does not account for the error between the empirical $\hat{\mathbb{P}}_N$ and the true $\mathbb{P}^*$ distributions. Moreover, the distribution $\bbP^*$ that generates the uncertainty samples may change at deployment time. This motivates the following distributionally robust chance-constrained formulation:
\begin{equation}
    V(\mathbf{0}) = 0;  \;\forall \bfx \neq \boldsymbol{0}, \; V(\bfx) - \epsilon\|\bfx\|_2^2 \geq 0 \; \text{and} \;
    \inf_{\mathbb{P}\in \calM_{N}^{r}}\bbP(-\dot{V}(\bfx,\bfxi) - \epsilon\|\bfx\|_2^2 \geq 0  ) \geq 1-\beta,  \label{eq: drccp_lf_vdot}
\end{equation}
where $\calM_{N}^{r}$ denotes the Wasserstein ambiguity set around the empirical distribution $\hat{\mathbb{P}}_N$ with user-defined radius $r$. Based on the discussion in Sec.~\ref{sec: prelim_cvar_drccp}, the following constraint is a sufficient condition for the distributionally robust chance constraint in \eqref{eq: drccp_lf_vdot} to hold, 
\begin{equation}
\label{eq: supinf_drccp}
    \sup_{\mathbb{P}\in \calM_{N}^{r}}\inf_{t \in \mathbb{R}}\left[\mathbb{E}_{\mathbb{P}}[G(\bfx,\bfxi)+t)_{+}] -t \beta \right] \leq 0, \quad \forall \bfx \neq \boldsymbol{0}. 
\end{equation}
%
%
As before,~\eqref{eq: supinf_drccp} is not amenable to a SOS formulation. The following result presents SOS conditions which are sufficient to ensure that \eqref{eq: supinf_drccp} holds. 

\begin{proposition}[DRCC-SOS Condition]\label{proposition: drcc_sos}
Assume $\beta \leq \frac{1}{N}$, consider the $1$-Wasserstein distance with $L_1$ norm as the metric $d$. The following is a sufficient condition for \eqref{eq: supinf_drccp} to hold,
\begin{equation}
\label{eq: drcc_alpha_relax}
    r \max_{1 \leq j \leq m} | \nabla V(\bfx)^\top d_j(\bfx) | + \max_{i} \beta (\dot{V}(\bfx,\bfxi_i) + \epsilon \|\bfx\|_2^2) \leq 0, \quad \forall \bfx \neq \bf0,
\end{equation}
where $\nabla V(\bfx)^\top d_j(\bfx)$ denotes the $j$-th element of the row vector. If $\Xi = \bbR^m$, then \eqref{eq: drcc_alpha_relax} is equivalent to \eqref{eq: supinf_drccp}. Also, if $f$ and $d_i$ are polynomials, \eqref{eq: drcc_alpha_relax} is implied by the following SOS conditions,
\begin{align}
    \pm  r \nabla V(\bfx)^\top d_j(\bfx) - \beta(\dot{V}(\bfx,\bfxi_i) - \epsilon\|\bfx\|_2^2) \in \text{SOS}(\bfx), \; \;
    \forall i = 1,2 \dots, N, \; \; \forall j = 1,2 \dots, m. \label{eq: proposition2_sos}
\end{align}

\end{proposition}
\begin{proof}
Based on \cite[Lemma V.8]{Hota2019DataDrivenCC} and \cite[Theorem 6.3]{Esfahani2018DatadrivenDR}, the supremum over the Wasserstein ambiguity set, i.e., condition \eqref{eq: supinf_drccp}, can be written conservatively as the sample average $\inf_{t \in \mathbb{R}}\left[\frac{1}{N} \sum_{i=1}^{N} (G(\bfx, \bfxi_i) + t)_+ - t \beta\right]$ and a regularization term $r L_G(\bfx)$, where $L_G(\bfx): \calX \mapsto \bbR_{>0}$ is the Lipschitz constant of $G(\bfx, \bfxi)$ in $\bfxi$. If $\Xi = \bbR^m$, then \eqref{eq: supinf_drccp} is equivalent to the sample average plus $r L_G(\bfx)$. Since the Lipschitz constant of a differentiable affine function equals the dual norm of its gradient, 
%
and the dual norm of the $L_1$ norm is the $L_{\infty}$ norm, we can define the convex function $L_G: \calX \mapsto \mathbb{R}_{>0}$ as
%
    $L_G(\bfx) = \|\nabla V(\bfx)^\top d(\bfx)\|_{\infty} = \max_{1 \leq j \leq m} | \nabla V(\bfx)^\top d_j(\bfx) |,$
%
which satisfies the property that $\bfxi \mapsto G (\bfx, \bfxi)$ is Lipschitz in $\bfxi$ with Lipschitz constant $L_G(\bfx)$. 
With the assumption that $\beta \leq \frac{1}{N}$, we use Proposition~\ref{proposition: chance_constraint_relax_lf} and conclude that \eqref{eq: drcc_alpha_relax} is a sufficient condition for \eqref{eq: supinf_drccp} and they are equivalent if $\Xi = \bbR^m$. Finally, inspired by the SOS relaxations of \eqref{eq: original_lf_condition} to \eqref{eq: original_lf_sos}, we can relax \eqref{eq: drcc_alpha_relax} to the $2Nm$ SOS constraints in~\eqref{eq: proposition2_sos}.
\end{proof}

Based on Proposition~\ref{proposition: drcc_sos}, we propose a DRCC-SOS formulation to find a Lyapunov function,
\begin{align}
\label{eq: drccp_sos_final}
    &V(\bfx) = \sum_{k=0}^{2d} c_k \bfx^k, \; c_0 = 0; \; \; V(\bfx) - \epsilon\|\bfx\|_2^2 \in \text{SOS}(\bfx); \\
    & \pm  r [\nabla V(\bfx)]^\top d_j(\bfx) - \beta(\dot{V}(\bfx,\bfxi_i) - \epsilon\|\bfx\|_2^2) \in \text{SOS}(\bfx), \; \;
    \forall i = 1,2 \dots, N, \; \; \forall j = 1,2 \dots, m. \notag
\end{align}
The next result identifies conditions under which the resulting Lyapunov function solves Problem~\ref{prob: LF_Construction_CCP}.

\begin{proposition}[Stability guarantee of DRCC-SOS formulation] 
\label{proposition: drcc_sos_guarantee}
Let the distribution $\bbP^*$ of $\bfxi$ in \eqref{eq: uncertain_closed_loop_system} be light-tailed, i.e., there exists an exponent $\rho$ such that $C := \mathbb{E}_{\bbP^*}[\exp(\| \bfxi \|^{\rho})] 
    < \infty$ .
Let the Wasserstein radius $r^*$ be given by:
\begin{equation} 
\label{r: estimate}
r^*_N(\alpha) := \begin{cases}
			\left( \frac{\log(c_1\alpha^{-1})}{c_2N} \right)^{1/\max(m,2)}, & N \geq \frac{\log(c_1\alpha^{-1})}{c_2}, \\
			\left( \frac{\log(c_1\alpha^{-1})}{c_2N} \right)^{1/\rho}, & N < \frac{\log(c_1\alpha^{-1})}{c_2},
			\end{cases}
\end{equation}
for $N\geq1$, $m \neq 2$, and $\alpha \in (0,1)$ being a user-specified risk parameter. The constants $c_1, c_2$ are positive and only depend on $\rho$, $C$ and $m$. Under those conditions, the Lyapunov function obtained from the DRCC-SOS formulation~\eqref{eq: drccp_sos_final} satisfies $\bbP^*(-\dot{V}(\bfx,\bfxi) - \epsilon\|\bfx\|_2^2 \geq 0) \geq (1-\alpha)(1-\beta)$.
\end{proposition}

\begin{proof}
For each $\bfx$, consider the events $A := \{\bbP^* \in \calM_N^{r^*}\}$ and $B := \{-\dot{V}(\bfx,\bfxi) - \epsilon\|\bfx\|_2^2 \geq 0\}$.

On the one hand, we have from \cite[Theorem 3.4]{Esfahani2018DatadrivenDR} that, under~\eqref{r: estimate}, $\bbP^*(A) \geq 1- \alpha$. 
On the other, from Proposition~\ref{proposition: drcc_sos}, the LF resulting from~\eqref{eq: drccp_sos_final} with $r^*$ satisfies $\inf_{\mathbb{P}\in \calM_{N}^{r^*}}\bbP(B) \geq 1-\beta$. Now, consider the probability of the event $B$ under the true distribution $\bbP^*$:
\begin{equation}
\bbP^*(B) \geq \bbP^*(B \cap A) = \bbP^*(B | A)\bbP^*(A) \geq \left(\inf_{\bbP \in \calM_{N}^{r_N(\alpha)}}\bbP(B)\right)\bbP^*(A) \geq (1-\alpha)(1-\beta) 
\end{equation}
\end{proof}

The DRCC-SOS formulation~\eqref{eq: drccp_sos_final} provides a stability guarantee (Proposition~\ref{proposition: drcc_sos_guarantee}) if there is no uncertainty distributional shift, i.e., $\bbP^*$ does not shift outside of $\calM_{N}^{r^*}$ at deployment time. However, similar to other SOS approaches, the formulation is restricted to polynomial systems and the non-existence of an SOS LF does not imply the non-existence of other valid LFs. This motivates us to consider next a more general candidate LF candidate, represented as a neural network.

\vspace*{-1ex}
\subsection{Neural Network Approach For Lyapunov Function Search}
\label{sec: nn_lf_search}

%
%
%
%

We propose a neural network approach that encourages the satisfaction of Lyapunov conditions by minimizing a loss function that quantifies their violation. 
Consider a neural network Lyapunov function (NN-LF) representation of the form $V_{\bftheta}(\bfx):= \| \phi_{\bftheta}(\bfx) -  \phi_{\bftheta}(\boldsymbol{0}) \|^2 + \hat{\alpha} \| \bfx \|^2$, where $\phi_{\bftheta}: \bbR^m \mapsto \bbR$ is a fully-connected neural network with parameters $\bftheta$ and $\tanh$ activations, and $\hat{\alpha}$ is a user-chosen parameter \citep{gaby2021lyapunov_net}. By construction, this function is positive definite and $V_{\bftheta}(\boldsymbol{0}) = 0$. We obtain a training set $\calD_{\text{LF}} := \{ \bfx_i\}_{i=1}^M$ by sampling uniformly from the domain of interest $\calX_{\delta}$ and then minimize the following empirical loss function:
\begin{equation}
\label{eq: lf_loss}
    \ell_{\text{LF}}(\bftheta) = \frac{1}{M}\sum_{i=1}^M (\dot{V}_{\bftheta}(\bfx_i) + \gamma \|\bfx_i\|)_+,
\end{equation}
%
where $\gamma$ is user-defined. This loss encourages a decrease of $V_{\bftheta}$ along the system trajectories.
%
%
To deal with the model uncertainty in \eqref{eq: uncertain_closed_loop_system}, we develop chance-constrained (CC) NN-LF  and distributionally robust chance-constrained (DRCC) NN-LF formulations. In both cases, we also have the offline uncertainty training set $\calD_{\xi} := \{ \bfxi_i\}_{i=1}^N$. For the CC-NN-LF formulation, we require
\begin{equation}\label{eq: cc_nn_lf_condition}
    \bbP^*(\dot{V}_{\bftheta}(\bfx, \bfxi) + \gamma \|\bfx \| \leq 0) \geq 1 - \beta,\quad \forall \bfx \in \calX_{\delta}.
\end{equation}
However, we are only given samples $\calD_{\xi}$ from $\bbP^*$. Assuming $\beta \leq \frac{1}{N}$, similarly to~Proposition~\ref{proposition: chance_constraint_relax_lf}, we approximate \eqref{eq: cc_nn_lf_condition} conservatively as, $\forall \bfx_i \in \calD_{\text{LF}}$, $\max_{j}(\dot{V}_{\bftheta}(\bfx_i, \bfxi_j) + \gamma \|\bfx_i \|) \leq 0$.
Thus, to aim for the satisfaction of \eqref{eq: cc_nn_lf_condition} for the training set $\calD_{\text{LF}}$, we construct the loss function,
\begin{equation}\label{eq: cc_lf_loss}
    \ell_{\text{CC-LF}}(\bftheta) = \frac{1}{M}\sum_{i=1}^M ( \max_{j}(\dot{V}_{\bftheta}(\bfx_i, \bfxi_j) + \gamma \|\bfx_i \|))_+.
\end{equation}
For the DRCC-NN-LF formulation, to account for errors between the empirical distribution $\hat{\bbP}_N$ and the true distribution $\bbP^*$ as well as possible distribution shift during deployment, we require: 
\begin{equation}\label{eq: drcc_nn_lf_condition}
    \inf_{\bbP \in \calM_{N}^{r}}\mathbb{P}(\dot{V}_{\bftheta}(\bfx, \bfxi) + \gamma \|\bfx\| \leq 0) \geq 1 - \beta, \quad \forall \bfx \in \calX_{\delta} .
\end{equation}
Note that \eqref{eq: drcc_nn_lf_condition} can be tightened in terms of the CVaR approximation as:
\begin{equation}
    \sup_{\bbP \in \calM_{N}^{r}} \inf_{t \in \mathbb{R}}\big[\bbE_{\bbP} (\dot{V}_{\bftheta}(\bfx,\bfxi) + \gamma \|\bfx\| +t)_{+} -t\beta \big] \leq 0, \; \; \forall \bfx \in \calX_{\delta}. 
\end{equation}
Next, using the uncertainty set $\calD_{\xi}$ and
the training dataset $\calD_{\text{LF}}$ and assuming $\beta \leq \frac{1}{N}$, similarly to Proposition.~\ref{proposition: drcc_sos}, we rewrite the inequality conservatively as (equivalently if $\Xi = \bbR^m$),
     $\forall \bfx_i \in \calD_{\text{LF}}$, $r \|\nabla V(\bfx_i)^\top d(\bfx_i)\|_{\infty} + \beta \max_{j}(\dot{V}_{\bftheta}(\bfx_i, \bfxi_j)) + \gamma \|\bfx_i\| \leq 0$.
Thus, we design the following empirical loss function for the DRCC-NN-LF formulation,
\begin{equation}\label{eq: dr_lf_loss}
    \ell_{\text{DRCC-LF}}(\bftheta) = \frac{1}{M}\sum_{i=1}^M (r \|\nabla V(\bfx_i)^\top d(\bfx_i)\|_{\infty} + \beta \max_{j}(\dot{V}_{\bftheta}(\bfx_i, \bfxi_j)) + \gamma \|\bfx_i \|)_+.
\end{equation}
The neural network approach, with the novel loss function designs in \eqref{eq: cc_lf_loss} and \eqref{eq: dr_lf_loss}, overcomes the issues noted above for the SOS approach. In particular, we do not require the dynamics to be described by polynomials and avoid scalability problems.

\begin{table}[t]
	\centering
	\caption{Comparison of Cases 1 and 2 under different online true distributions. Here, ``vio. rate'' denotes violation rate: (validations with $\dot{V} > 0$)/(total validations), and ``vio. area'' denotes average violation area over all simulations: (data points with $\dot{V} > 0$)/(total data points).
    5000 realizations of the online true uncertainty $\bfxi^*$ are sampled from uniform and Gaussian distributions: $\bfxi^* \sim [\calU(1,4),\calU(1,2)]^{\top}$ and $\bfxi^* \sim [\calN(4,1.5),\calN(1,1.5)]^{\top}$ for Case 1, $\bfxi^* \sim [\calU(5,7),\calU(-1,1)]^{\top}$ and $\bfxi^* \sim [\calN(7,1),\calN(1,1)]^{\top}$ for~Case~2.
    }
	\scalebox{.92}{
    \begin{tabular}{ccccccccc}
		\toprule 
        \multirow{ 2}{*}{Formulations} & \multicolumn{ 2}{c}{Case 1 Uniform} & \multicolumn{ 2}{c}{Case 1 Gaussian} & \multicolumn{ 2}{c}{Case 2 Uniform} & \multicolumn{ 2}{c}{Case 2 Gaussian} \\
         & vio. rate & vio. area & vio. rate & vio. area & vio. rate & vio. area & vio. rate & vio. area \\
		\midrule  
		SOS & 14.28\% & 0.94\% & 12.14\% & 1.53\% & 100\% & 15.52\% & 100\% & 18.55\% \\
        CC-SOS & 11.78\% & 0.89\% & 8.30\% & 1.24\% & 0.00\% & 0.00\% & 5.10\% & 0.04\% \\
        DRCC-SOS & 0.02\% & 0.00\% & 5.24\% & 0.80\% & 0.00\% & 0.00\% & 1.64\% & 0.01\% \\
	 NN & 31.80\% &           1.95\% & 16.66\% &      1.65\% &       100\% & 17.10\% & 100\% &     19.53\% \\
        CC-NN & 1.82\% & 0.01\% & 6.24\% & 0.72\% & 0.00\% & 0.00\% & 1.26\% & 0.01\% \\
         DRCC-NN & 0.00\% & 0.00\% & 3.22\% & 0.38\% & 0.00\% & 0.00\% & 0.72\% & 0.00\% \\
		\bottomrule
	\end{tabular}}
\label{table: numerical_1}
\vspace*{-1ex}
\end{table}

\vspace*{-1ex}
\section{Evaluation}
\label{sec: evaluation}

We apply the SOS approach (Sec.~\ref{sec: drccp_sos_lf}) and the neural network approach
(Sec.~\ref{sec: nn_lf_search}) to synthesize LFs for a polynomial system and a pendulum system under model uncertainty.

\setcounter{subfigure}{0}
\begin{figure}[t]
	\centering
    \subfigure[\scriptsize Original SOS Search]{\includegraphics[width=.3\linewidth]{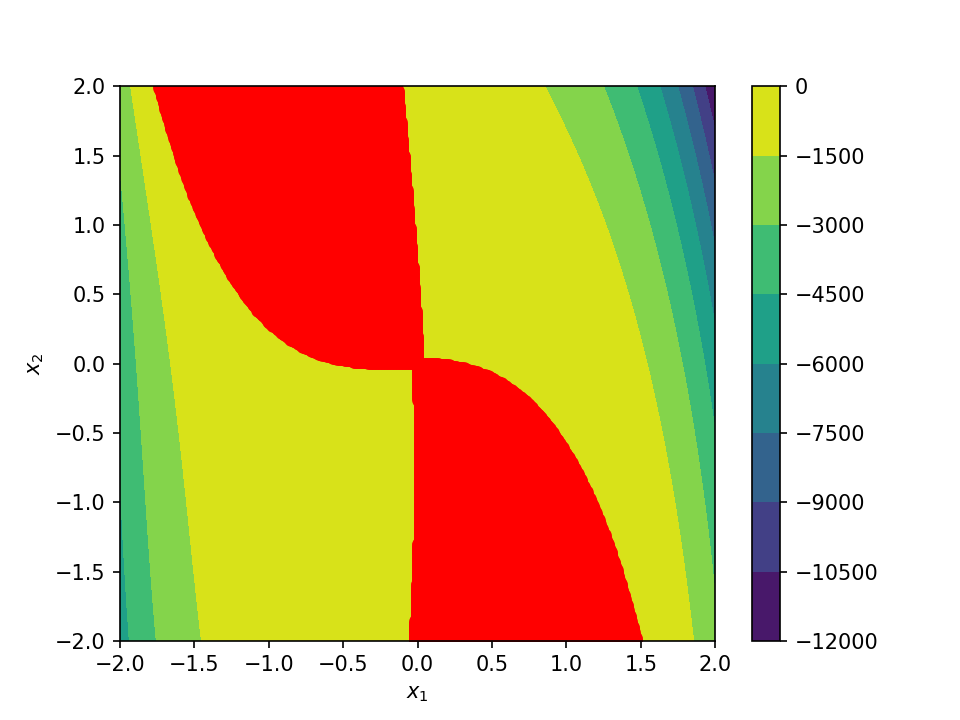}\label{fig: case_2_sos}}
	\subfigure[\scriptsize CC-SOS Search ]{\includegraphics[width=.3\linewidth]{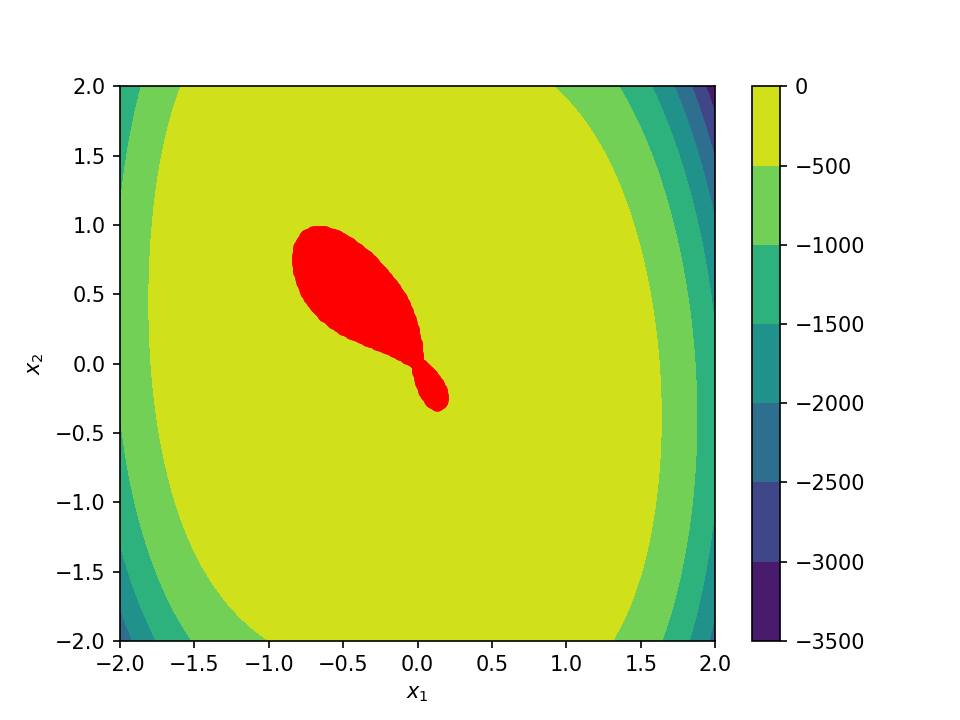}\label{fig: case_2_ccp}}
	\subfigure[\scriptsize DRCC-SOS Search ]{\includegraphics[width=.3\linewidth]{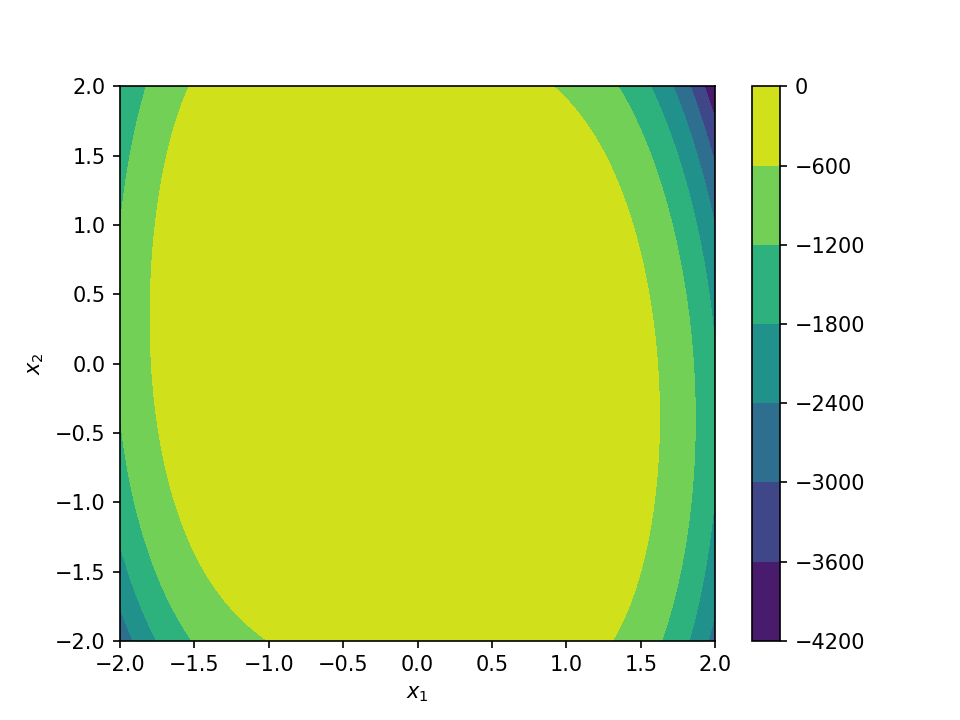}\label{fig: case_2_drccp}}
    \\[-1ex]
    \subfigure[\scriptsize Original NN Search]{\includegraphics[width=.3\linewidth]{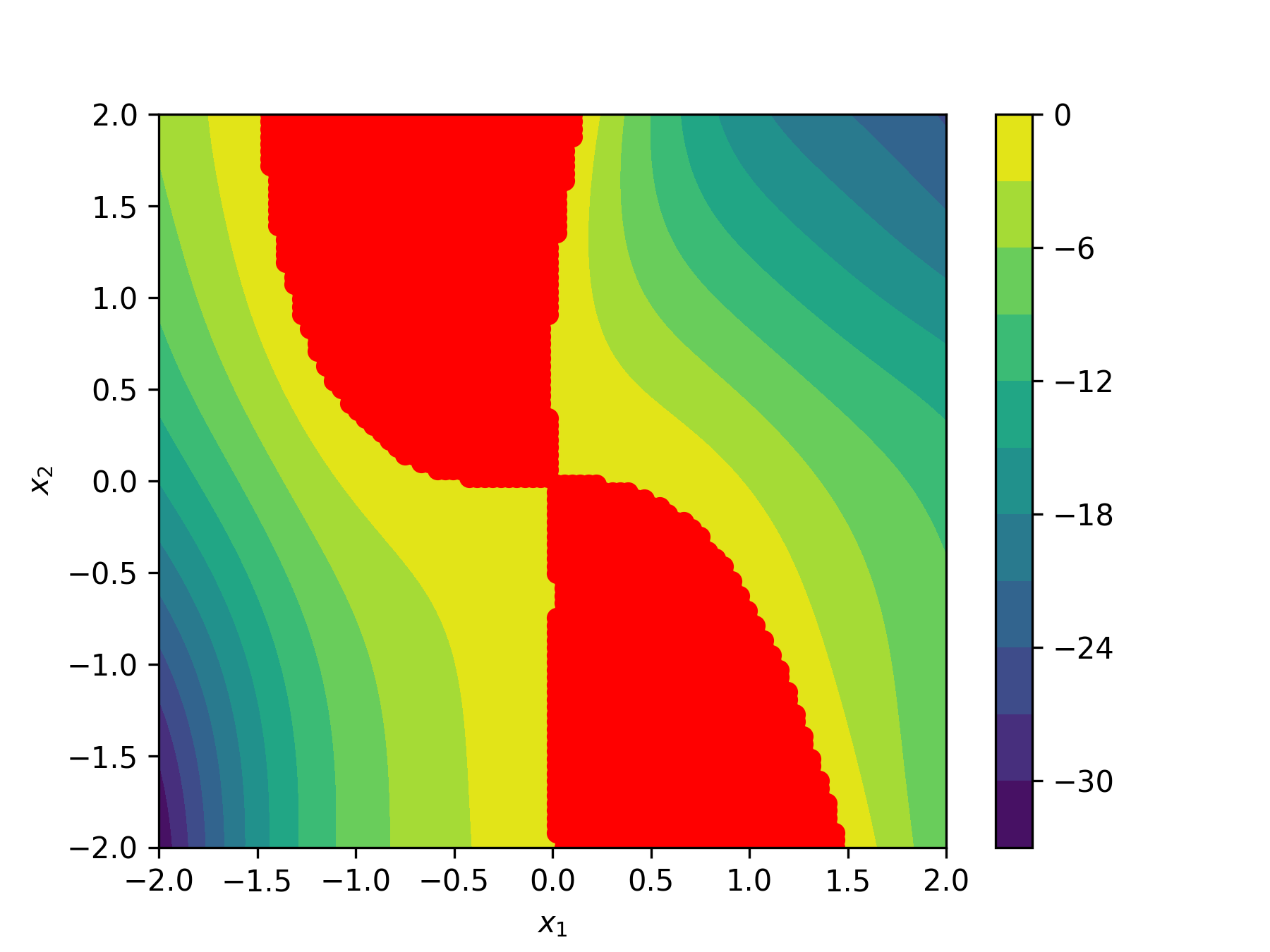}\label{fig: case_2_nn}}
    \subfigure[\scriptsize CC-NN Search ]{\includegraphics[width=.3\linewidth]{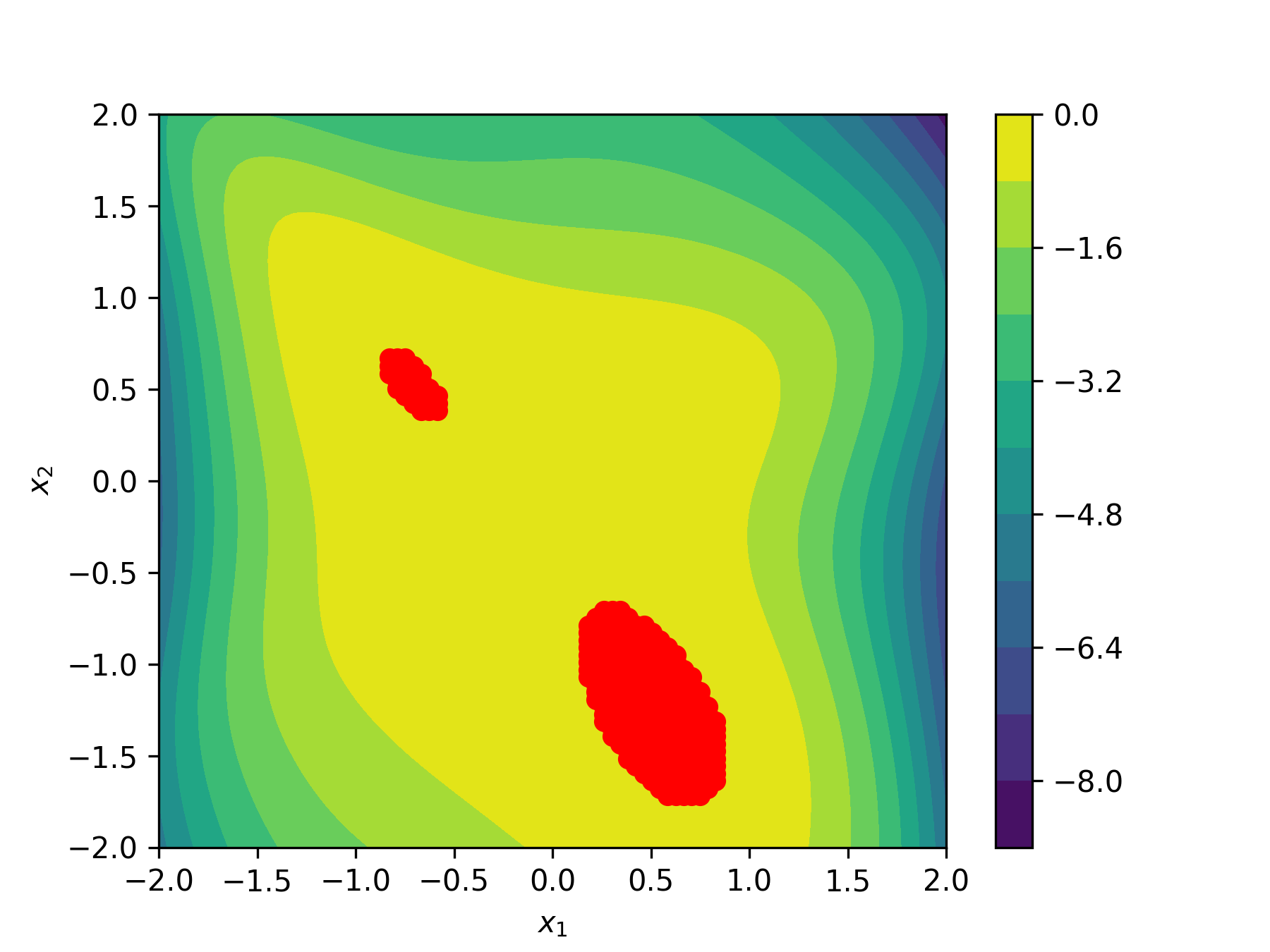}\label{fig: case_2_cc_nn}}
    \subfigure[\scriptsize DRCC-NN Search ]{\includegraphics[width=.3\linewidth]{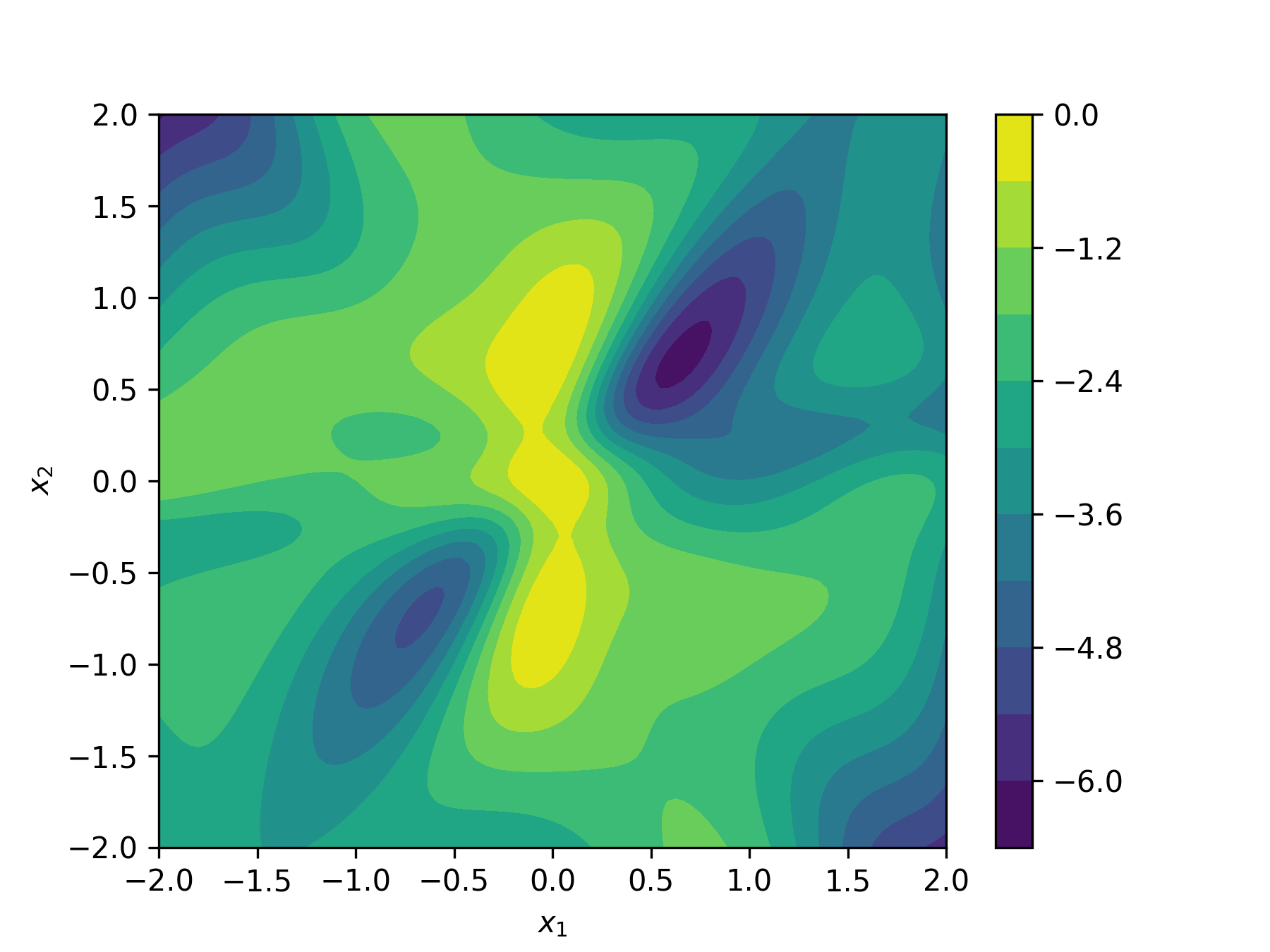}\label{fig: case_2_drcc_nn}}
	\caption{Results from SOS and NN formulations to design LF certificates for the polynomial system with Case 2 perturbations and online uncertainty $\bfxi^* = [1.9, 3.0]^\top$. The plots display the value of $\dot{V}$ over the domain, where the red areas indicate positive values (violation of the LF derivative requirements).}
	\label{fig: lf_all_compare}
\end{figure}

\vspace{1ex}
\noindent
\textbf{Third-degree Polynomial System:}
Consider a two-dimensional polynomial system \citep{mit_slides}:
%
\begin{equation}\label{eq: exp_1_dynamic}
    \left[ \begin{matrix}
    \dot{x}_1 \\
    \dot{x}_2 
    \end{matrix} \right]
	=	
	\left[ \begin{matrix}
			-\frac{1}{2}x_1^3 - \frac{3}{2}x_1^2 - x_2  \\
			6x_1 - x_2
			\end{matrix} \right] + \sum_{i = 1}^2 d_i(\bfx)\xi_i,
\end{equation}
%
with two cases for the model uncertainty:
\begin{itemize}[nosep,leftmargin=2em]
    \item Case 1: $r = 0.25, \ d_1(\bfx) = -[x_1,x_2]^{\top}\!\!, \ d_2(\bfx) = -[x_2,0]^{\top}\!\!, \ \bfxi \sim [\calN(5,1),\calN(3,1)]^{\top}$.
    \item Case 2: $r = 0.15,  \ d_1(\bfx) = -[(x_1^3+x_2),x_2]^{\top}\!\!, \ d_2(\bfx) = -[x_2,x_1]^{\top}\!\!, \ \bfxi \sim [\calN(6,1),\calN(0,1)]^{\top}$. 
\end{itemize}
Suppose that $9$ samples $\{ \bfxi_i \}_{i = 1}^{9}$ are available offline and set the confidence level to $\beta = 0.1$.

We compare the SOS search results with polynomial degree of $4$ for the original SOS formulation in \eqref{eq: original_lf_sos}, the CC-SOS formulation in \eqref{eq: ccp_sos_final}, and the DRCC-SOS formulation in \eqref{eq: drccp_sos_final}. We also include results from the NN formulation in \eqref{eq: lf_loss}, the CC-NN formulation in \eqref{eq: cc_lf_loss}, and the DRCC-NN formulation in \eqref{eq: dr_lf_loss}. 
For the neural network approach discussed in Sec.~\ref{sec: nn_lf_search}, we parametrize $V_{\bftheta}(\bfx) = | \phi_{\bftheta}(\bfx) -  \phi_{\bftheta}(\bf0) | + \hat{\alpha} \| \bfx \|$, where $\phi_{\bftheta}(\bfx)$ is a fully connected three-layer neural network with 2-D input, two 16-D hidden layers, and 1-D output, 
with $\tanh$ activations. We train the network with the ADAM optimizer \citep{Adam} with learning rate $0.005$ and Xavier initializer, and set the parameter $\hat{\alpha} = 0.05$.

We report qualitative results in Fig.~\ref{fig: lf_all_compare} for Case 2 with online uncertainty $\bfxi^* = [1.9, 3.0]^\top$. We uniformly sample $\{\bfx_i\}_{i=1}^{5000}$ states in the region $x_1, x_2 \in [-2,2]$. For the first-column plots, the resulting LFs from the baseline SOS and NN formulation fail to satisfy the Lyapunov condition for uncertain systems of the form \eqref{eq: exp_1_dynamic}, and the violation area is large since neither formulation takes uncertainty into account. For the second-column plots, the resulting LF from the CC-SOS or CC-NN formulation is less sensitive to uncertainty, since both take offline uncertainty samples into account. However, the resulting $V$ still fails to satisfy the Lyapunov condition for \eqref{eq: exp_1_dynamic}. The LF resulting from our DRCC-SOS and DRCC-NN formulations in the last column satisfies the Lyapunov conditions for \eqref{eq: exp_1_dynamic}, even with out-of-distribution uncertainty. 
Table~\ref{table: numerical_1} shows quantitative results. We report the violation rate and average violation area for each of the $6$ formulations: in all cases, the DRCC formulations outperform the CC and baseline formulations (no uncertainty considered) using either the SOS or the neural network approach in terms of violation rate and mean violation area.

\vspace{1ex}
\noindent
\textbf{Pendulum:}
Consider a pendulum with angle $\theta$ and angular velocity $\dot{\theta}$ following dynamics:
%
%
%
\begin{equation}\label{eq: pendulum_dynamics}
    \left[ \begin{matrix}
    \dot{\theta} \\
    \ddot{\theta}
    \end{matrix} \right]
	= \left[ \begin{matrix}
    \dot{\theta} \\
    \frac{-mgl \sin{\theta} - b \dot{\theta}}{ml^2}
    \end{matrix} \right] + \left[ \begin{matrix}
    0 & 0 \\
    -\frac{0.05b\dot{\theta}}{ml^2} & -\frac{0.05mgl\sin{\theta}}{ml^2}
    \end{matrix} \right] \bfxi,
\end{equation}
where $g = 9.81$ is the gravity acceleration, $m = 1.0$ is the ball mass, $l = 0.5$ is the length, $b = 0.1$ is the damping, and $d(\bfx) = [d_1(\bfx), d_2(\bfx)]$ is the perturbation matrix with $d_1$ and $d_2$ representing perturbations in damping and length, respectively.
%
%
%
%
We use $3$ offline uncertainty samples $\{ \bfxi_i \}_{i = 1}^{3}$ with $\bfxi_i \sim [\calN(0,1),\calN(0,1)]^{\top}$, and set the confidence level $\beta = 0.1$. The SOS search polynomial is set to have a degree $4$ with Wasserstein radius $r = 0.03$. The neural network $\phi_{\bftheta}$ consists of a fully connected four-layer architecture, featuring a 3-D input, three 64-D hidden layers, and a 2-D output. The network employs $\tanh$ activations, and the pendulum state $\theta$ is rewritten as two states, $\sin{\theta}$ and $\cos{\theta}$. The Wasserstein radius is set to $r = 0.12$. We train the network with the ADAM optimizer with learning rate $0.002$ and Xavier initializer, and set the parameter $\hat{\alpha} = 0.5$.

We compare the qualitative results between the SOS-based approaches and the NN-based approaches in Fig.~\ref{fig: lf_compare_pendulum} with the online uncertainty  $\bfxi^* = [-3.6,1.4]^{\top}$. Similar to Fig.~\ref{fig: lf_all_compare}, only the DRCC-SOS and DRCC-NN formulations meet the Lyapunov conditions within the domain of interest. The derivative violations observed near the small neighborhood of the equilibrium in the DRCC-NN formulation are a common issue in neural network-based Lyapunov functions, as reported in previous studies~\citep{gaby2021lyapunov_net, Chang2019NeuralLC}.

\begin{figure}[t]
	\centering
    \subfigure[\scriptsize Original SOS Search ]{\includegraphics[width=.3\linewidth]{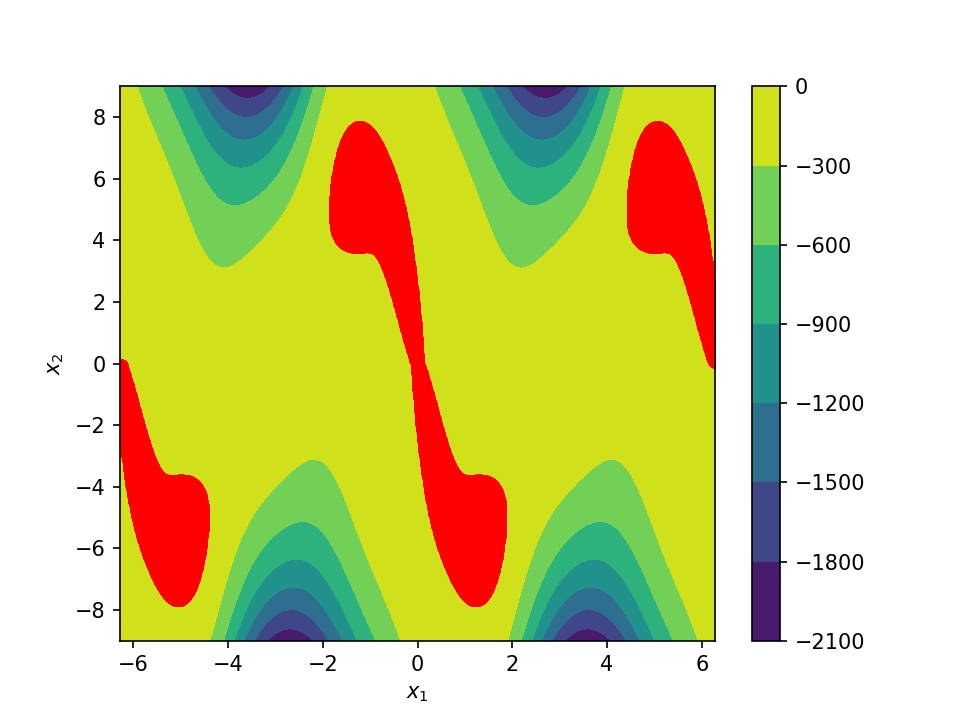}\label{fig: case_3_sos}}
	\subfigure[\scriptsize CC-SOS Search  ]{\includegraphics[width=.3\linewidth]{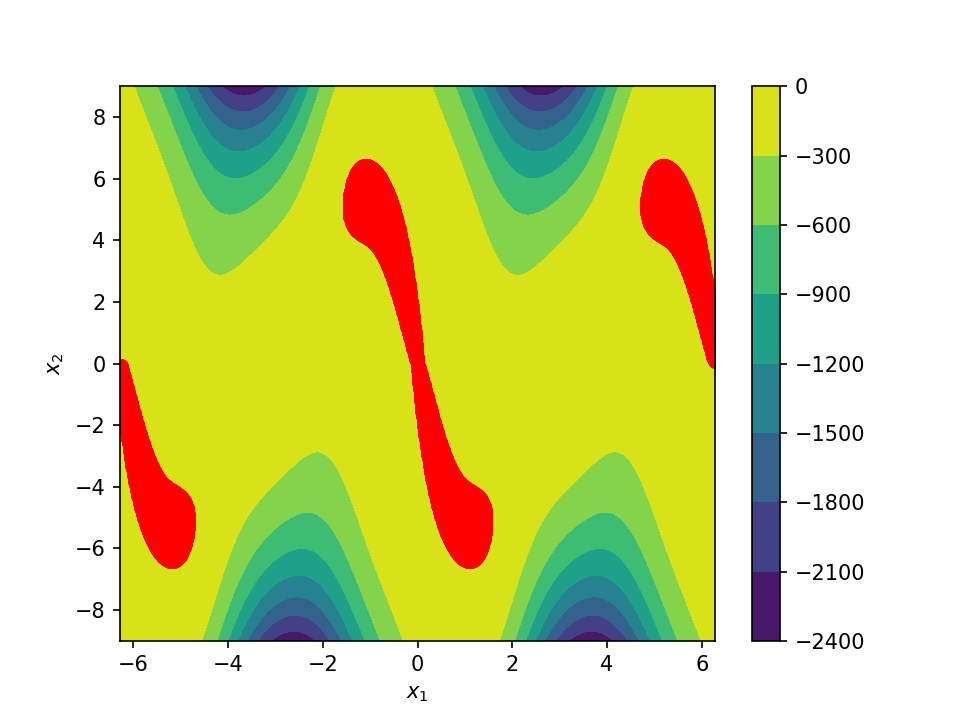}\label{fig: case_3_ccp}}
	\subfigure[\scriptsize DRCC-SOS Search  ]{\includegraphics[width=.3\linewidth]{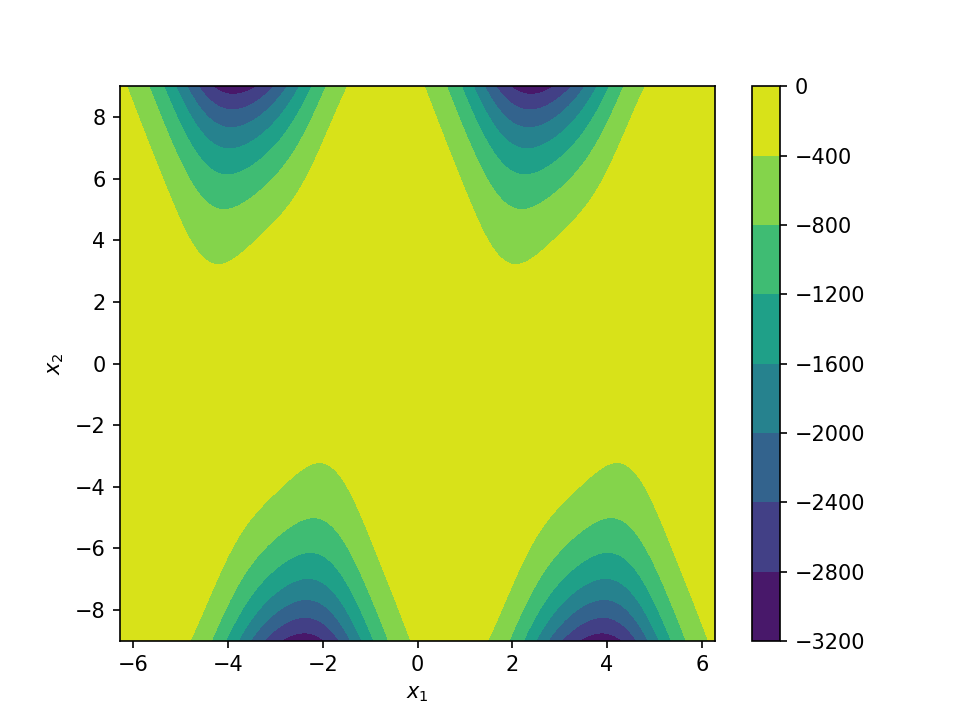}\label{fig: case_3_drccp}}
     \\[-1ex]
    \subfigure[\scriptsize Original NN Search  ]{\includegraphics[width=.3\linewidth]{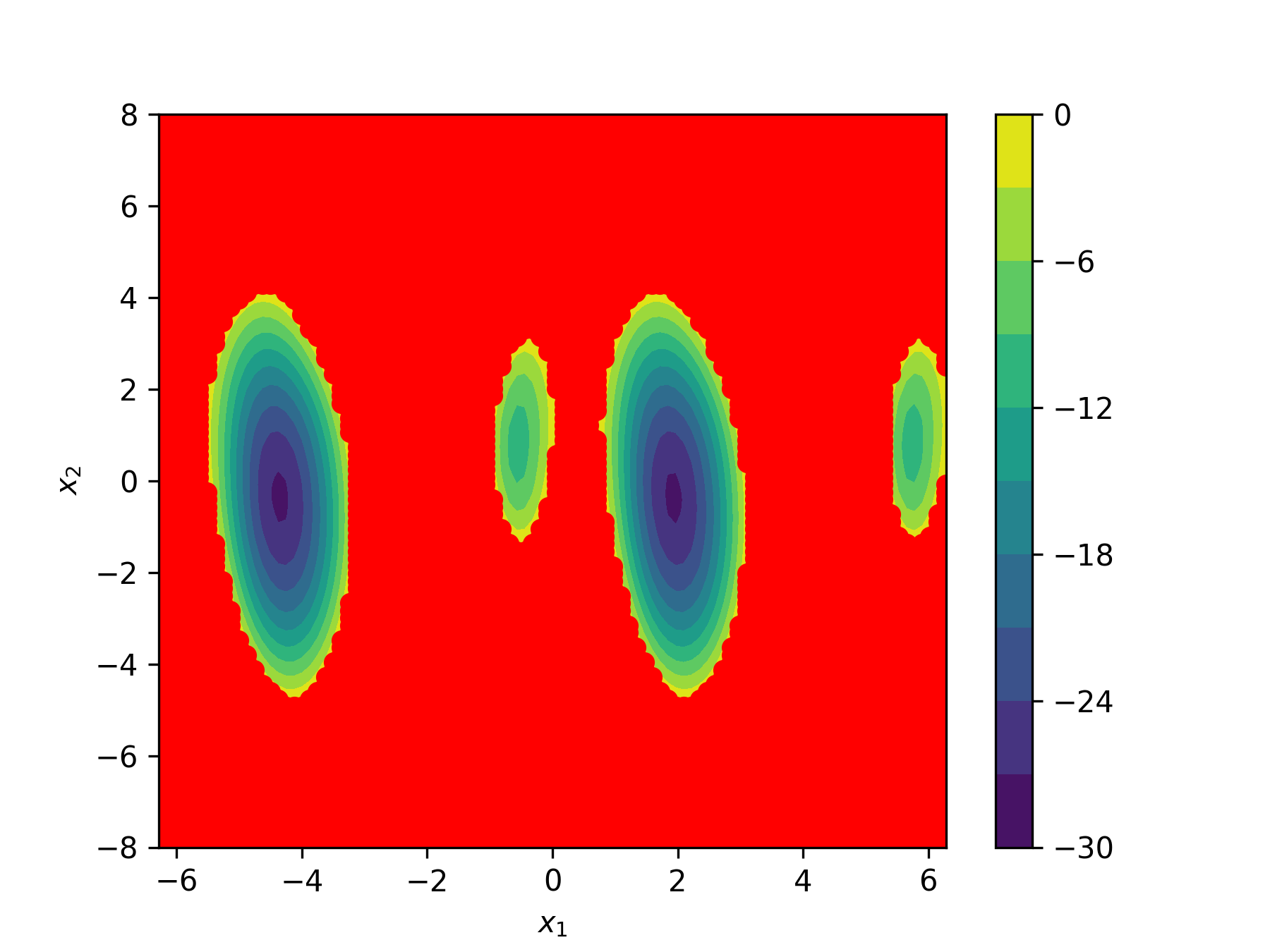}\label{fig: case_3_nn}}
	\subfigure[\scriptsize CC-NN Search ]{\includegraphics[width=.3\linewidth]{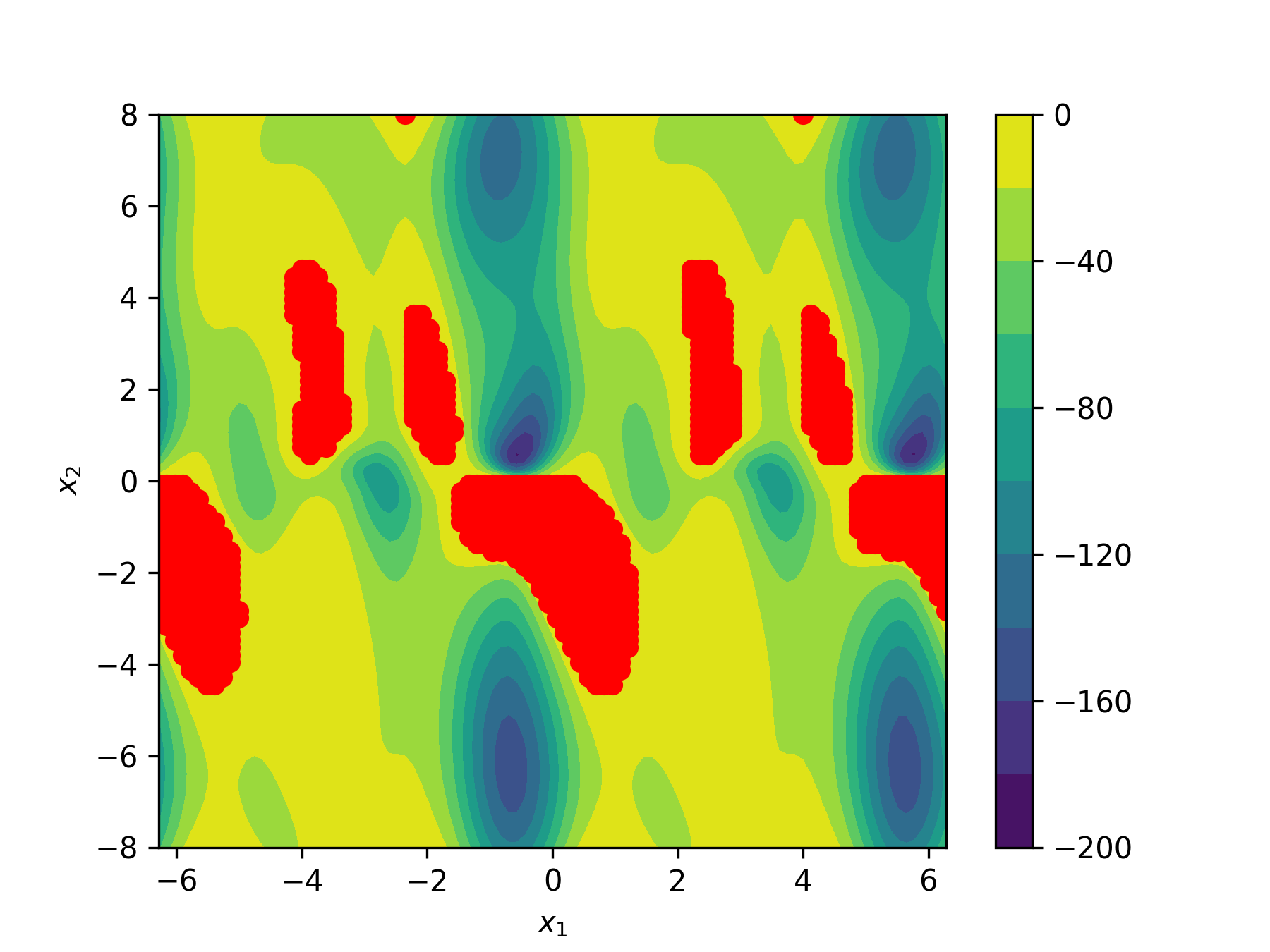}\label{fig: case_3_cc_nn}}
	\subfigure[\scriptsize DRCC-NN Search  ]{\includegraphics[width=.3\linewidth]{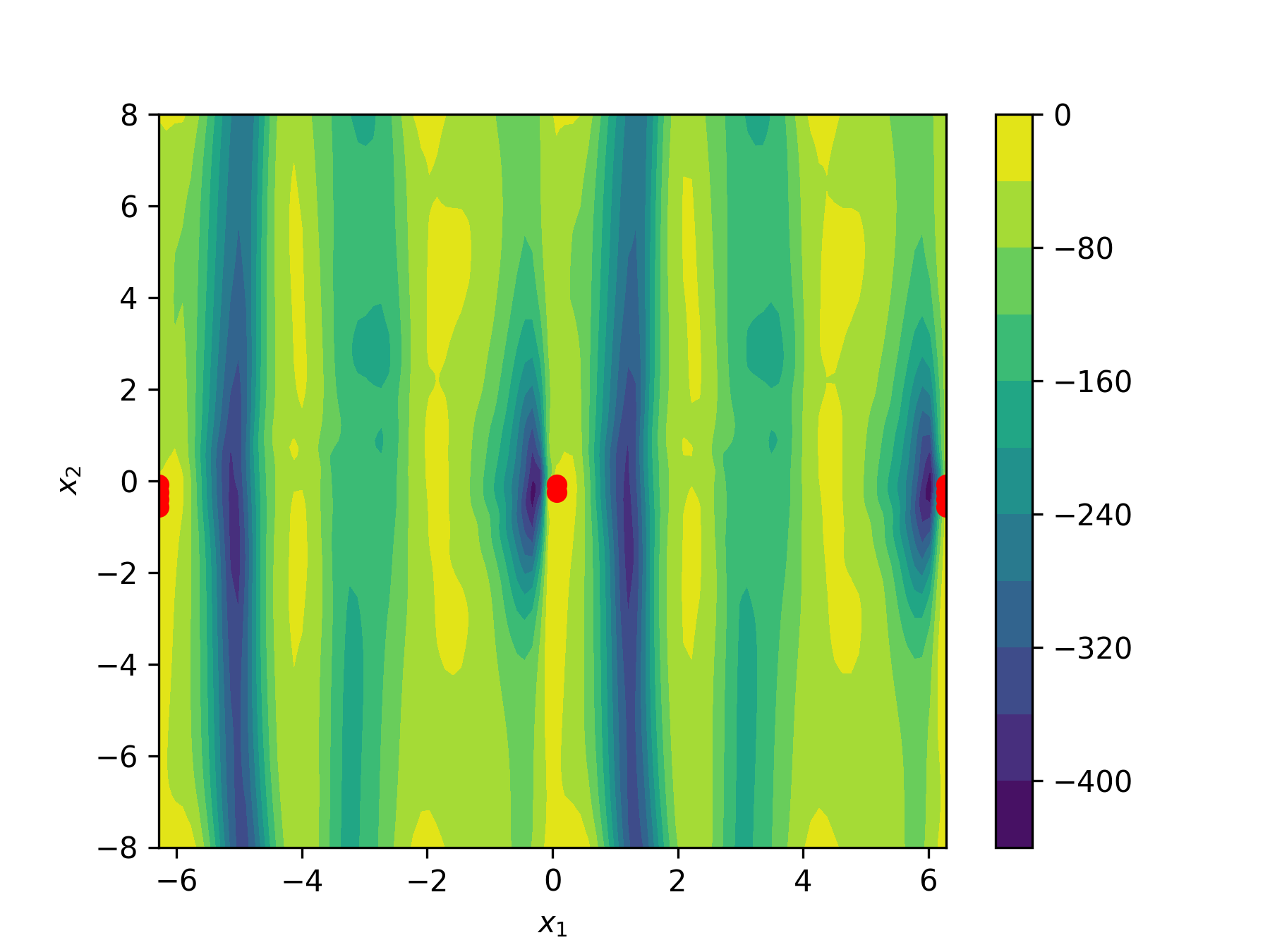}\label{fig: case_3_drcc_nn}}
	\caption{Results from SOS and NN formulations to design LF certificates for a pendulum with perturbation in the damping and length and online uncertainty $\bfxi^* = [-3.6,1.4]^{\top}$. The plots display the value of $\dot{V}$ over the domain, where the red areas indicate positive values (violation of the LF derivative requirement).}
    \label{fig: lf_compare_pendulum}
    \vspace*{-3ex}
\end{figure}

\section{Conclusions}

We investigated the synthesis of Lyapunov functions for uncertain closed-loop dynamical systems. With only finitely many offline uncertainty samples, we derived novel distributionally robust formulations of sum-of-squares and neural-network approaches. The evaluation shows that LFs learned with our DRCC formulations are valid even for out-of-sample model errors. Future work will consider joint CLF and control policy search under distributional uncertainty and applications to higher-dimensional robotic control systems.


\acks{The authors gratefully acknowledge support from NSF under grants IIS-2007141 and CCF-2112665.}


\bibliography{ref}

\end{document}